\documentclass[11pt]{article}
\usepackage{amssymb,amsmath,amsthm}
\usepackage{url}
\usepackage{xspace}

\usepackage{hyperref}

\usepackage[symbol]{footmisc}

\newtheorem{thm}{Theorem}

\newtheorem{lem}[thm]{Lemma}

\def\R{\mbox{\ensuremath{\mathbb R}}\xspace}

\def\C{\mbox{\ensuremath{\mathcal C}}\xspace}
\def\F{\mbox{\ensuremath{\mathcal F}}\xspace}
\def\H{\mbox{\ensuremath{\mathcal H}}\xspace} 

\def\eps{\mbox{\ensuremath{\varepsilon}}\xspace}

\mathcode`l="8000
\begingroup
\makeatletter
\lccode`\~=`\l
\DeclareMathSymbol{\lsb@l}{\mathalpha}{letters}{`l}
\lowercase{\gdef~{\ifnum\the\mathgroup=\m@ne \ell \else \lsb@l \fi}}%
\endgroup

\begin{document}

\title{Radon numbers grow linearly}
\author{D\"om\"ot\"or P\'alv\"olgyi\footnote{MTA-ELTE Lend\"ulet Combinatorial Geometry Research Group, Institute of Mathematics, E\"otv\"os Lor\'and University (ELTE), Budapest, Hungary.
Research supported by the Lend\"ulet program of the Hungarian Academy of Sciences (MTA), under grant number LP2017-19/2017.}}

\maketitle

\begin{abstract}
Define the $k$-th Radon number $r_k$ of a convexity space as the smallest number (if it exists) for which any set of $r_k$ points can be partitioned into $k$ parts whose convex hulls intersect.
Combining the recent abstract fractional Helly theorem of Holmsen and Lee with earlier methods of Bukh, we prove that $r_k$ grows linearly, i.e., $r_k\le c(r_2)\cdot k$.
\end{abstract}

\medskip

\section{Introduction}

Define a \emph{convexity space} as a pair $(X,\C)$, where $X$ is any set of points and $\C$, the collection of convex sets, is any family over $X$ that contains $\emptyset, X$, and is closed under intersection.
The convex hull, $conv(S)$, of some point set $S\subset X$ is defined as the intersection of all convex sets containing $S$, i.e., $conv(S)=\cap\{C\in\C\mid S\subset C\}$; since \C is closed under intersection, $conv(S)$ is the minimal convex set containing $C$.
This generalization of convex sets includes several examples; for an overview, see the book by van de Vel \cite{Vel} or for a more recent work, \cite{MoranY}.
It is a natural question what properties of convex sets of $\R^d$ are preserved or what the relationships are among them for general convexity spaces.
A much investigated function is the \emph{Radon number} $r_k$ (sometimes also called \emph{partition number} or \emph{Tverberg number}), which is defined as the smallest number (if it exists) for which any set of $r_k$ points can be partitioned into $k$ parts whose convex hulls intersect.
For $k=2$, we simply write $r=r_2$.

In case of the convex sets of $\R^d$, it was shown by Radon \cite{Radon} that $r=d+2$ and by Tverberg \cite{Tverberg} that $r_k=(d+1)(k-1)+1$.
Calder \cite{Calder} and Eckhoff \cite{Eckhoff79} raised the question whether $r_k\le (r-1)(k-1)+1$ also holds for general convexity spaces (when $r$ exists), and this became known as Eckhoff's conjecture.
It was shown by Jamison \cite{Jamison} that the conjecture is true if $r=3$, and that the existence of $r$ always implies that $r_k$ exists and $r_k\le r^{\lceil \log_2 k\rceil}\le (2k)^{\log_2 r}$.
His proof used the recursion $r_{kl}\le r_kr_l$ which was later improved by Eckhoff \cite{Eckhoff00} to $r_{2k+1}\le (r-1)(r_{k+1}-1)+r_k+1$, but this did not significantly change the growth rate of the upper bound.
Then Bukh \cite{Bukh} has disproved the conjectured $r_k\le (r-1)(k-1)+1$ by showing an example where $r=4$, but $r_k\ge 3k-1$ (just one more than the conjectured value) and has also improved the upper bound to $r_k= O(k^2\log^2 k)$, where the hidden constant depends on $r$.
We improve this to $r_k=O(k)$, which is optimal up to a constant factor and might lead to interesting applications.

\begin{thm}\label{main}
	If a convexity space $(X,\C)$ has Radon number $r$, then $r_k\le c(r)\cdot k$.
\end{thm}

Our proof combines the methods of Bukh with recent results of Holmsen and Lee \cite{HolmsenLee}.
In particular, we will use the following version of the classical fractional Helly theorem \cite{KL}.

\begin{thm}[Holmsen-Lee \cite{HolmsenLee}]\label{fh}
	If a convexity space $(X,\C)$ has Radon number $r$, then there is an $f$ such that for any $\alpha>0$ there is a $\beta>0$ such that for any finite family \F of convex sets if at least an $\alpha$ fraction of the $f$-tuples of \F are intersecting, then a $\beta$ fraction of \F intersects.
\end{thm}

There are several other connections between the parameters of a convexity space \cite{Vel}; for example, earlier it was already shown \cite{Levi} that in convexity spaces the Helly number is always strictly less than $r$, while in \cite{HolmsenLee} it was also shown that the colorful Helly number \cite{Barany} can be also bounded by some function of $r$ (and this implied Theorem \ref{fh} combined with a combinatorial result from \cite{Holmsen}).\footnote{We would like to point out that a difficulty in proving these results is that the existence of a Carath\'eodory-type theorem is not implied by the existence of $r$.}
It was also shown in \cite{HolmsenLee} that it follows from the work of Alon et al.~\cite{AKMM} that
weak \eps-nets \cite{ABFK} of size $c(\eps,r)$ also exist and a $(p,q)$-theorem \cite{AK} also holds, so understanding these parameters better might lead to improved \eps-net bounds.
It remains an interesting challenge and a popular topic to find new connections among such theorems;
for some recent papers studying the Radon numbers or Tverberg theorems of various convexity spaces, see \cite{dLLHRS17a,dLLHRS17b,dLHMM,FGKVW,Letzter,Patak,Patakova,Soberon}, while for a comprehensive survey, see B\'ar\'any and Sober\'on \cite{BS}.


\bigskip
\noindent
\textbf{\large Restricted vs.\ multiset} 
\smallskip

In case of general convexity spaces, there are two, slightly different definitions of Radon numbers (\cite{Vel}: 5.19).
When in the point set $P$ to be partitioned we do not allow repetitions, i.e., $P$ consist of \emph{different} points, the parameter is called \emph{restricted} Radon number, which we will denote by $r_k^{(1)}$.
If repetitions are also allowed, i.e., we want to partition a multiset, the parameter is called \emph{unrestricted} or \emph{multiset} Radon number, which we will denote by $r_k^{(m)}$.
The obvious connection between these parameters is $r_k^{(1)}\le r_k^{(m)}\le (k-1)(r_k^{(1)}-1)+1$.
In the earlier papers multiset Radon numbers were preferred, while later papers usually focused on restricted Radon numbers; we followed the spirit of the age, so the results in the Introduction were written using the definition of $r_k^{(1)}$, although some of the bounds (like Jamison's or Eckhoff's) are valid for both definitions.
The proof of Theorem \ref{main}, however, also works for multisets, so we will in fact prove the stronger $r_k^{(m)}=O(k)$, and in the following simply use $r_k$ for the multiset Radon number $r_k^{(m)}$.

A similar issue arises in Theorem \ref{fh};
is \F allowed to be a multifamily?
Though not emphasized in \cite{HolmsenLee}, their proof also works in this case and we will use it for a multifamily.
Note that this could be avoided with some cumbersome tricks, like adding more points to the convexity space without increasing the Radon number $r$ to make all sets of a family different, but we do not go into details, as Theorem \ref{fh} anyhow holds for multifamilies.

\section{Proof}

Fix $r$, and a collection of points $P$ with cardinality $tk$, where we allow repetitions and the cardinality is understood as the sum of the multiplicities.
We will treat all points of $P$ as if they were different even if they coincide in $X$, e.g., when taking subsets.

We need to show that if $t\ge c(r)$, then we can partition $P$ into $k$ sets whose convex hulls intersect.
For a fixed constant $s$, define \F to be the family of convex sets that are the convex hull of some $s$-element subset of $P$,
i.e., $\F=\{conv(S)\mid S\subset P, |S|=s\}$.
Since we treat all points of $P$ as different, \F will be a multifamily with $|\F|=\binom{tk}s$.
We will refer to the point set $S$ whose convex hull gave some $F=conv(S)\in\F$ as the \emph{vertices} of $F$ (despite that some of the points might be in the interior of $F$).
Note that for some $S\ne S'$, we might have $conv(S)=conv(S')$, but the vertices of $conv(S)$ and $conv(S')$ will still be $S$ and $S'$; since $P$ is a multiset, it is even possible that $S\cap S'=\emptyset$.

The constants $t$ and $s$ will be set to be large enough compared to some parameters that we get from Theorem \ref{fh} when we apply it to a fixed $\alpha$.
(Our arguments work for any $0<\alpha<1$.)
First we set $s$ to be large enough depending on $\alpha$ and $r_f$ (recall that $r_f\le r^{\log f}$ is a constant \cite{Jamison}),
then $t$ to be large enough depending on $s$ and $\beta$ (that belongs to our chosen $\alpha$).
In particular, we can set $s=\log(\frac 1{1-\alpha_s})r_ff^{fr_f}$ and $t=\max(\frac{s^2}{\beta};\frac{(fs)^2}{k(1-\alpha_t)})$, where $0<\alpha_s,\alpha_t<1$ are any two numbers such that $\alpha_s\cdot\alpha_t=\alpha$.
Also, we note that the proof from \cite{Holmsen,HolmsenLee} gives $f\le r^{r^{\log r}}$ and $\beta=\Omega(\alpha^{r^f})$ for Theorem \ref{fh}.
Combining all these would give an upper bound around of ${r^{r^{r^{\log r}}}}$ for $t$.

Theorem \ref{main} will be implied by the following lemma and Theorem \ref{fh}.

\begin{lem}\label{fhholds}
	An $\alpha$ fraction of the $f$-tuples of \F are intersecting.
\end{lem}
\begin{proof}	
	Since $t$ is large enough, almost all $f$-tuples will be vertex-disjoint, thus it will be enough to deal with such $f$-tuples.
	More precisely, the probability of an $f$-tuple being vertex-disjoint is at least $(1-\frac{fs}{tk})^{fs}\ge 1-\frac{(fs)^2}{tk}\ge \alpha_t$ by the choice of $t$.
	We need to prove that at least an $\alpha_s$ fraction of these vertex-disjoint $f$-tuples will be intersecting.
	
	Partition the vertex-disjoint $f$-tuples into groups depending on which $(fs)$-element subset of $P$ is the union of their vertices.
	We will show that for each group an $\alpha_s$ fraction of them are intersecting.
	We do this by generating the $f$-tuples of a group uniformly at random and show that such a random $f$-tuple will be intersecting with probability at least $\alpha_s$.
	For technical reasons, suppose that $m=\frac s{r_f}$ is an integer and partition the $fs$ supporting points of the group randomly into $m$ subsets of size $fr_f$, denoted by $V_1,\ldots,V_m$.
	Call an $f$-tuple \emph{type} $(V_1,\ldots,V_m)$ if each set of the $f$-tuple intersects each $V_i$ in $r_f$ points.
	Since these $V_i$ were picked randomly, it is enough to show that the probability that a $(V_1,\ldots,V_m)$-type $f$-tuple is intersecting is at least $\alpha_s$.
	
 	The $(V_1,\ldots,V_m)$-type $f$-tuples can be uniformly generated by partitioning each $V_i$ into $f$ equal parts of size $r_f$.
	Therefore, it is enough to show that such a random $f$-tuple will be intersecting with probability at least $\alpha_s$.
	Since $|V_i|\ge r_f$, there is at least one partition of the first $r_f$ points of $V_i$ to $f$ parts whose convex hulls intersect.
	Since we can distribute the remaining $(f-1)r_f$ points of $V_i$ to make all $f$ parts equal, we get that when we partition $V_i$ into $f$ equal parts of size $r_f$, the convex hulls of these parts will intersect with probability at least $\binom{fr_f}{r_f,r_f,\ldots,r_f}^{-1}\ge f^{-fr_f}$.
	Since these events are independent for each $i$, we get that the final $f$-tuple will be intersecting with probability at least $1-(1-f^{-fr_f})^m\ge 1-e^{-mf^{-fr_f}}\ge \alpha_s$ by the choice of $s$.
\end{proof}

Therefore, if $s$ is large enough, the conditions of Theorem \ref{fh} are met, so at least $\beta\binom{tk}s$ members of \F intersect.
In other words, these intersecting sets form an $s$-uniform hypergraph \H on $tk$ vertices that is $\beta$-dense.
We need to show that \H has $k$ disjoint edges to obtain the desired partition of $P$ into $k$ parts with intersecting convex hulls.
For a contradiction, suppose that \H has only $k-1$ disjoint edges.
Then every other edge meets one of their $(k-1)s$ vertices.
There are at most $(k-1)s\binom{tk}{s-1}$ such edges, which is less than $\beta\binom{tk}s$ if $(k-1)s<\beta\frac{tk-s+1}{s}$, but this holds 
by the choice of $t$.
This finishes the proof of Theorem \ref{fh}.\qed

\bigskip
\noindent
\textbf{\large Concluding remarks} 
\smallskip

It is an interesting question to study how big $f$ can be compared to $r$ and the Helly number $h$ of $(X,\C)$.
The current bound \cite{HolmsenLee} gives $f\le h^{r_h}\le r^{r^{\log r}}$.
We would like to point out that the first inequality, $f\le h^{r_h}$, can be (almost) strict, as shown by the following example, similar to Example 3 (cylinders) of \cite{MoranY}.
Let $X=\{1,\ldots, q\}^d$ be the points of a $d$-dimensional grid, and let \C consist of all axis-parallel affine subspaces.
(Note that for $q=2$, $X$ will be the vertices of a $d$-dimensional cube, and \C its faces.)
It is easy to check that $h=2$, $r=\lfloor \log(d+1)+2\rfloor$ and $f=d+1$; the last equality follows from that for $\alpha=\frac{d!}{d^d}$ we need $\beta=\frac 1q$ when \F consists of all $qd$ axis-parallel affine hyperplanes (if $q$ is large enough).

It is tempting to assume that Theorem \ref{main} would improve the second inequality, $h^{r_h}\le r^{r^{\log r}}$, as instead of $r_h\le r^{\log h}$ we can use $r_h=O(h)$.
Unfortunately, recall that the hidden constant depended on $r$, in particular, it is around ${r^{r^{r^{\log r}}}}$.
We have a suspicion that this might not be entirely sharp, so a natural question is whether this dependence could be removed to improve $r_k\le {r^{r^{r^{\log r}}}}\cdot k$ to $r_k\le c \cdot r\cdot k$.
This would truly lead to an improvement of the upper bound on $f$ in Theorem \ref{fh} and would be enough for several applications \cite{BS}.

\bigskip
\noindent
\textbf{\large Acknowledgment} 
\smallskip

I would like to thank Boris Bukh and Narmada Varadarajan for discussions on \cite{Bukh}, Andreas Holmsen for calling my attention to the difference between restricted and multiset Radon numbers, espcially for confirming that Theorem \ref{fh} also holds for multisets, and G\'abor Dam\'asdi, Bal\'azs Keszegh, Padmini Mukkamala and G\'eza T\'oth for feedback on earlier versions of this manuscript, especially for fixing the computations in the proof of Lemma \ref{fhholds}.

\end{document}